\documentclass[12pt,a4paper]{article}
\usepackage[utf8]{inputenc}
\usepackage{graphicx}
\usepackage{amsmath}
\usepackage[shortlabels]{enumitem}
\usepackage{float}

\usepackage{geometry}
\usepackage{listings}
\usepackage{amssymb}
\usepackage[parfill]{parskip}
\usepackage[table]{xcolor}
\usepackage{longtable}
\usepackage{booktabs}
\usepackage{setspace}
\usepackage[T1]{fontenc}
\usepackage[format=hang]{subcaption}
\usepackage[format=hang]{caption}

\usepackage{mathtools}

\usepackage{amsthm}

\newtheorem{theorem}{Theorem}[section]

\newtheorem{proposition}{Proposition}[section]
\newtheorem{lemma}{Lemma}[section]
\theoremstyle{remark}
\newtheorem{remark}{Remark}[section]

\newcommand{\R}{\mathbb{R}}

\newcommand{\N}{\mathbb{N}}

\newcommand{\C}{\mathbb{C}}

\providecommand{\abs}[1]{\left|#1\right|}

\newcommand{\PAR}[1]{\left(#1\right)}

\newcommand{\BPAR}[1]{\left[ #1\right]}
\newcommand{\DMAT}[1]{\left|\begin{matrix}#1\end{matrix}\right|}

\newcommand{\MAT}[1]{\left[\begin{matrix}#1\end{matrix}\right]}
\newcommand{\MINOR}[1]{\left(\begin{matrix}#1\end{matrix}\right)}

\newcommand{\diag}{\operatorname{diag}}

\newcommand{\THREEFIG}[5]{
  \begin{figure}[H]
   \centering
   \begin{subfigure}[t]{0.3\linewidth}
    \includegraphics[width=1\linewidth]{#1_a.png}
    \caption[]{#3} 
    \label{fig:#1_a}
   \end{subfigure}
   \begin{subfigure}[t]{0.3\linewidth}
    \includegraphics[width=1\linewidth]{#1_b.png}
    \caption[]{#4} 
    \label{fig:#1_b}
   \end{subfigure}
   \begin{subfigure}[t]{0.3\linewidth}
    \includegraphics[width=1\linewidth]{#1_c.png}
    \caption[]{#5} 
    \label{fig:#1_c}
   \end{subfigure}
   \begin{subfigure}[t]{0.06\linewidth}
    \includegraphics[width=1\linewidth]{#1_cbar.png}
   \end{subfigure}
   \caption[]{#2}
   \label{fig:#1}
  \end{figure}
}

\newcommand{\TWOFIG}[4]{
  \begin{figure}[H]
   \centering
   \begin{subfigure}[t]{0.446\linewidth}
    \includegraphics[width=1\linewidth]{#1_a.png}
    \caption[]{#3} 
    \label{fig:#1_a}
   \end{subfigure}
   \begin{subfigure}[t]{0.446\linewidth}
    \includegraphics[width=1\linewidth]{#1_b.png}
    \caption[]{#4} 
    \label{fig:#1_b}
   \end{subfigure}
   \begin{subfigure}[t]{0.089\linewidth}
    \includegraphics[width=1\linewidth]{#1_cbar.png}
   \end{subfigure}
   \caption[]{#2}
   \label{fig:#1}
  \end{figure}
}

\usepackage[strict]{changepage}
\usepackage[absolute,overlay]{textpos}
\usepackage[hyphens]{url}
\usepackage{hyperref}
\hypersetup{colorlinks=false, pdfborder={0 0 0},}

\author{Karl Lundengård, Jonas Österberg, Sergei Silvestrov\\Mälardalen University}

\begin{document}
\title{Extreme points of the Vandermonde determinant on the sphere and some limits involving the generalized Vandermonde determinant}
\maketitle

\begin{abstract}
The values of the determinant of Vandermonde matrices with real elements are analyzed both visually and analytically over the unit sphere in various dimensions.
For three dimensions some generalized Vandermonde matrices are analyzed visually.
The extreme points of the ordinary Vandermonde determinant on finite-dimensional unit spheres are given as the roots of rescaled Hermite polynomials and a recursion relation is provided for the polynomial coefficients.
Analytical expressions for these roots are also given for dimension three to seven.
A transformation of the optimization problem is provided and some relations between the ordinary and generalized Vandermonde matrices involving limits are discussed.
\end{abstract}

\newpage

\tableofcontents

\newpage

%%%%%%%%%%%%%%%%%%%%%%%%%%%%%%%%%%%%%%%%%%%%%%%%%%%%%%%%%%%%%%%%%%%%%%%%%%%%%%%%%%%
%%%%%%%%%%%%%%%%%%%%%%%%%%%%%%%%%%%%%%%%%%%%%%%%%%%%%%%%%%%%%%%%%%%%%%%%%%%%%%%%%%%
%%%%%%%%%%%%%%%%%%%%%%%%%%%%%%%%%%%%%%%%%%%%%%%%%%%%%%%%%%%%%%%%%%%%%%%%%%%%%%%%%%%
%%%%%%%%%%%%%%%%%%%%%%%%%%%%%%%%%%%%%%%%%%%%%%%%%%%%%%%%%%%%%%%%%%%%%%%%%%%%%%%%%%%

\section{Introduction}

Here a \textit{generalized Vandermonde matrix} is determined by two vectors $\vec x_n=(x_1,\cdots,x_n)\in K^n$ and $\vec a_m=(a_1,\cdots,a_m)\in K^m$, where $K$ is usually the real or complex field, and is defined as
\begin{align} \label{def:gn}
 G_{mn}(\vec x_n,\vec a_m)=\BPAR{x_j^{a_i}}_{mn}.
\end{align}
For square matrices only one index is given, $G_n \equiv G_{nn}$.

Note that the term \emph{generalized Vandermonde matrix} has been used for several kinds of matrices that are not equivalent to \eqref{def:gn}, see \cite{kalman} for instance.

As a special case we have the ordinary Vandermonde matrices
\begin{align*}
  V_{mn}(\vec x_n)& =G_{mn}(\vec x_n,(0,1,\cdots,m-1))=\BPAR{x_j^{i-1}}_{mn} \\
                & =\MAT{
	                        1     &     1     & \cdots &   1    \\
	                       x_1    &    x_2    & \cdots &  x_n   \\
	                      \vdots  &  \vdots   & \ddots & \vdots \\
	                    x_1^{m-1} & x_2^{m-1} & \cdots & x_n^{m-1}}.
\end{align*}
Note that some authors use the transpose as the definition and possibly also let indices run from $0$.
All entries in the first row of Vandermonde matrices are ones and by considering $0^0=1$ this is true even when some $x_j$ is zero.
\\
In this article the following notation will sometimes be used:
\[ \vec{x}_I = (x_{i_1},x_{i_2},\ldots,x_{i_n}),~I = \{ i_1, i_2, \ldots, i_n \}.  \]
For the sake of convenience we will use $\vec x_n$ to mean $\vec{x}_{I_n}$ where $I_n = \{ 1, 2, \ldots, n \}$.
\\
\begin{proposition} \label{prop:vandermondeDeterminant}
The determinant of (square) Vandermonde matrices has the well known form
\[ v_n\equiv v_n(\vec{x}_n)\equiv \det V_n(\vec{x}_n)=\prod_{1\leq i<j\leq n}(x_j-x_i). \]
\end{proposition}
This determinant is also simply referred to as the \textit{Vandermonde determinant}.
The determinant of generalized Vandermonde matrices 
\[ g_n\equiv g_n(\vec{x}_n,\vec{a}_n)\equiv \det G_n(\vec{x}_n,\vec{a}_n) \] 
seem to defy such a simple statement in the general case, a partial treatment can be found in \cite{ernst}.

Vandermonde matrices have uses in polynomial interpolation.
The coefficients of the unique polynomial $c_0+c_1x+\cdots+c_{n-1}x^{n-1}$ that passes through $n$ points $(x_i,y_i)\in\C^2$ with distinct $x_i$ are
\[
 \MAT{c_0&c_1&\cdots&c_{n-1}}
 =
 \MAT{y_1&y_2&\cdots&y_n}
 \MAT{
	1 & 1 & \cdots & 1\\
	x_1 & x_2 & \cdots & x_n\\
	\vdots & \vdots & \ddots & \vdots\\
	x_1^{n-1} & x_2^{n-1} & \cdots & x_n^{n-1}
}^{-1}.
\]
In this context the $x_i$ are called \textit{nodes}.
Some other uses are in differential equations \cite{kalman} and time series analysis \cite{klein}.

In Section \ref{sec:visual3d} we introduce the reader to the behavior of the Vandermonde determinant $v_3(\vec{x}_3)=(x_3-x_2)(x_3-x_1)(x_2-x_1)$ by some introductory visualizations. We also consider $g_3$ for some choices of exponents.

In Section \ref{sec:analytical} we optimize the Vandermonde determinant $v_n$ over the unit sphere $S^{n-1}$ in $\mathbb{R}^n$ finding a rescaled Hermite polynomial whose roots give the extreme points.

In Section \ref{sec:analytical_v3} we arrive at the results for the special case $v_3$ in a slightly different way.

In Section \ref{sec:transformation} we present a transformation of the optimization problem into an equivalent form with known optimum value.

In Section \ref{sec:visualnd} we extend the range of visualizations to $v_4,\cdots,v_7$ by exploiting some analytical results.

In Section \ref{sec:power_limit} we prove some limits involving the the generalized Vandermonde matrix and determinant.

% Here we present background and uses of Vandermonde in general and the optimization in particular,\\
% We comment on the organization of this article.
% introduce orthogonal and symmetric polynomials\\

\subsection{Visual exploration in 3D} \label{sec:visual3d}

In this section we plot the values of the determinant 
\[ v_3(\vec x_3)=(x_3-x_2)(x_3-x_1)(x_2-x_1), \]
and also the generalized Vandermonde determinant $g_3(\vec x_3,\vec a_3)$ for three different choices of $\vec a_3$
over the unit sphere $x_1^2 + x_2^2 + x_3^2=1$ in $\mathbb{R}^3$.
Our plots are over the unit sphere but the determinant exhibits the same general behavior over centered spheres of any radius.
This follows directly from \ref{def:gn} and that exactly one element from each row appear in the determinant.
For any scalar $c$ we get
\begin{align*}
 g_n(c\vec x_n,\vec a_n)= \left[\prod_{i=1}^nc^{a_i}\right] g_n(\vec x_n,\vec a_n),
\end{align*}
which for $v_n$ becomes
\begin{align} \label{eq:vn_scalingx}
v_n(c\vec x_n)=c^{\displaystyle\frac{n(n-1)}{2}} v_n(\vec x_n),
\end{align}
and so the values over different radii differ only by a constant factor.

\THREEFIG{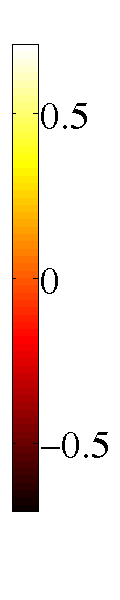}{Plot of $v_3(\vec{x}_3)$ over the unit sphere.}
                      {Plot with respect to the regular $\vec{x}$-basis.}
                      {Plot with respect to the $\vec{t}$-basis, see \eqref{eq:xt_transform}.}
                      {Plot with respect to parametrization \eqref{eq:t_spherical}.}

In Figure \ref{fig:g3_012plots} we have plotted the value of $v_3(\vec{x}_3)$ over the unit sphere and traced curves where it vanishes as black lines.
The coordinates in Figure \ref{fig:g3_012plots_b} are related to $\vec{x}_3$ by
\begin{align}
  \vec x_3=
  \MAT{2 & 0 & 1\\-1&1&1\\-1&-1&1}
  \MAT{1/\sqrt{6}&0&0\\0&1/\sqrt{2}&0\\0&0&1/\sqrt{3}}
  \vec t \label{eq:xt_transform},
\end{align}
where the columns in the product of the two matrices are the basis vectors in $\R^3$.
The unit sphere in $\mathbb{R}^3$ can also be described using spherical coordinates. In Figure \ref{fig:g3_012plots_c} the following convention was used.
\begin{equation}
  \vec t(\theta,\phi)=\MAT{
  \cos(\phi) \sin(\theta)\\
  \sin(\phi)\\
  \cos(\phi) \cos(\theta)\\
  }. \label{eq:t_spherical}
\end{equation}
We will use this $\vec{t}$-basis and spherical coordinates throughout this section.

From the plots in Figure \ref{fig:g3_012plots} it can be seen that the number of extreme points for $v_3$ over the unit sphere seem to be $6=3!$.
It can also been seen that all extreme points seem to lie in the plane through the origin and orthogonal to an apparent symmetry axis in the direction $(1,1,1)$, the direction of $t_3$.
We will see later that the extreme points for $v_n$ indeed lie in the hyperplane $\sum_{i=1}^n x_i=0$ for all $n$, see Theorem \ref{thm:hyperplane}, and the number of extreme points for $v_n$ count $n!$, see Remark \ref{rem:n_factorial}.

The black lines where $v_3(\vec{x}_3)$ vanishes are actually the intersections between the sphere and the three planes $x_3-x_1=0$, $x_3-x_2=0$ and $x_2-x_1=0$, as these differences appear as factors in $v_3(\vec{x}_3)$.

We will see later on that the extreme points are the six points acquired from permuting the coordinates in \label{rem:permutations3D}
\begin{align*}
 \vec{x}_3=\frac{1}{\sqrt{2}} \left(-1,0,1\right).
\end{align*}
For reasons that will become clear in Section \ref{sec:roots_by_pol} it is also useful to think about these coordinates as the roots of the polynomial
\begin{align*}
 P_3(x) = x^3-\frac{1}{2}x.
\end{align*}

So far we have only considered the behavior of $v_3(\vec{x}_3)$, that is $g_3(\vec{x}_3,\vec a_3)$ with $\vec a_3=(0,1,2)$.
We now consider three generalized Vandermonde determinants, namely $g_3$ with $\vec a_3=(0,1,3)$, $\vec a_3=(0,2,3)$ and $\vec a_3=(1,2,3)$.
These three determinants show increasingly more structure and they all have a neat formula in terms of $v_3$ and the elementary symmetric polynomials
\begin{align*}
 e_{kn}=e_k(x_1,\cdots,x_n)=\sum_{1\leq i_1<i_2<\cdots<i_k\leq n} x_{i_1}x_{i_2}\cdots x_{i_k},
\end{align*}
where we will simply use $e_k$ whenever $n$ is clear from the context.

\THREEFIG{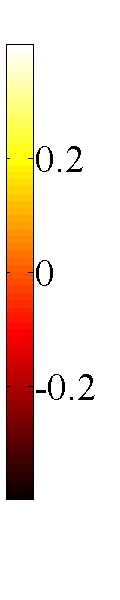}{Plot of $g_3(\vec{x}_3,(0,1,3))$ over the unit sphere.}
                      {Plot with respect to the regular $\vec{x}$-basis.}
                      {Plot with respect to the $\vec{t}$-basis, see \eqref{eq:xt_transform}.}
                      {Plot with respect to angles given in \eqref{eq:t_spherical}.}

In Figure \ref{fig:g3_013plots} we see the determinant
\begin{align*}
 g_3(\vec{x}_3,(0,1,3))=\DMAT{1&1&1\\x_1&x_2&x_3\\x_1^3&x_2^3&x_3^3}=v_3(\vec{x}_3)e_{1},
\end{align*}
plotted over the unit sphere.
The expression $v_3(\vec{x}_3)e_{1}$ is easy to derive, the $v_3(\vec{x}_3)$ is there since the determinant must vanish whenever any two columns are equal, which is exactly what the Vandermonde determinant expresses. The $e_1$ follows by a simple polynomial division.
As can be seen in the plots we have an extra black circle where the determinant vanishes compared to Figure \ref{fig:g3_012plots}.
This circle lies in the plane $e_1=x_1+x_2+x_3=0$ where we previously found the extreme points of $v_3(\vec{x}_3)$ and thus doubles the number of extreme points to $2\cdot 3!$.

A similar treatment can be made of the remaining two generalized determinants that we are interested in, plotted in the following two figures.

\THREEFIG{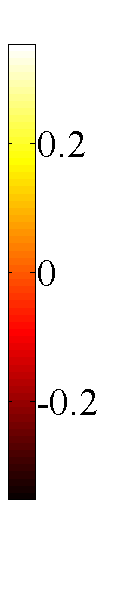}{Plot of $g_3(\vec{x}_3,(0,2,3))$ over the unit sphere.}
                      {Plot with respect to the regular $\vec{x}$-basis.}
                      {Plot with respect to the $\vec{t}$-basis, see \eqref{eq:xt_transform}.}
                      {Plot with respect to parametrization \eqref{eq:t_spherical}.}
\THREEFIG{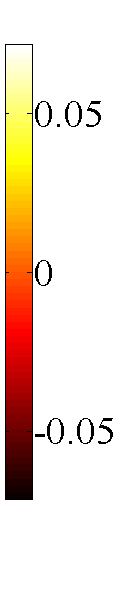}{Plot of $g_3(\vec{x}_3,(1,2,3))$ over the unit sphere.}
                      {Plot with respect to the regular $\vec{x}$-basis.}
                      {Plot with respect to the $\vec{t}$-basis, see \eqref{eq:xt_transform}.}
                      {Plot with respect to parametrization \eqref{eq:t_spherical}.}

The four determinants treated so far are collected in Table \ref{tab:g3determinantlist}.
\begin{table}
 \centering
 \begin{tabular}{cl}
  \toprule
  $\vec a_3$    & $g_3(\vec{x}_3,\vec a_3)$\\\midrule
  $(0,1,2)$   & $v_3(\vec{x}_3)e_0=(x_3-x_2)(x_3-x_1)(x_2-x_1)$\\
  $(0,1,3)$   & $v_3(\vec{x}_3)e_1=(x_3-x_2)(x_3-x_1)(x_2-x_1)(x_1+x_2+x_3)$\\
  $(0,2,3)$   & $v_3(\vec{x}_3)e_2=(x_3-x_2)(x_3-x_1)(x_2-x_1)(x_1x_2+x_1x_3+x_2x_3)$\\
  $(1,2,3)$   & $v_3(\vec{x}_3)e_3=(x_3-x_2)(x_3-x_1)(x_2-x_1)x_1x_2x_3$ \\
  \bottomrule
 \end{tabular}
 \caption{Table of some determinants of generalized Vandermonde matrices.}
 \label{tab:g3determinantlist}
\end{table}
Derivation of these determinants is straight forward.
We note that all but one of them vanish on a set of planes through the origin.
For $\vec a=(0,2,3)$ we have the usual Vandermonde planes but the intersection of $e_2=0$ and the unit sphere occur at two circles.
\begin{align*}
 x_1x_2+x_1x_3+x_2x_3
 &=\frac{1}{2}\PAR{ (x_1+x_2+x_3)^2 - (x_1^2+x_2^2+x_3^2)}\\
 =\frac{1}{2}\PAR{ (x_1+x_2+x_3)^2 - 1 }
 &=\frac{1}{2}\PAR{ x_1+x_2+x_3 + 1 }\PAR{ x_1+x_2+x_3 - 1 },
\end{align*}
and so $g_3(\vec x_3,(0,2,3))$ vanish on the sphere on two circles lying on the planes $x_1+x_2+x_3 + 1=0$ and $x_1+x_2+x_3 - 1=0$.
These can be seen in Figure \ref{fig:g3_023plots} as the two black circles perpendicular to the direction $(1,1,1)$.

Note also that while $v_3$ and $g_3(\vec x_3,(0,1,3))$ have the same absolute value on all their respective extreme points (by symmetry) we have that both $g_3(\vec x_3,(0,2,3))$ and $g_3(\vec x_3,(1,2,3))$ have different absolute values for some of their respective extreme points.

% Here we present a visualization of the determinant over sphere for generalized and ordinary.
% We note that the points lies in the 1,1,1,... hyperplane. -> forward to proof later

%%%%%%%%%%%%%%%%%%%%%%%%%%%%%%%%%%%%%%%%%%%%%%%%%%%%%%%%%%%%%%%%%%%%%%%%%%%%%%%%%%%
%%%%%%%%%%%%%%%%%%%%%%%%%%%%%%%%%%%%%%%%%%%%%%%%%%%%%%%%%%%%%%%%%%%%%%%%%%%%%%%%%%%
%%%%%%%%%%%%%%%%%%%%%%%%%%%%%%%%%%%%%%%%%%%%%%%%%%%%%%%%%%%%%%%%%%%%%%%%%%%%%%%%%%%
%%%%%%%%%%%%%%%%%%%%%%%%%%%%%%%%%%%%%%%%%%%%%%%%%%%%%%%%%%%%%%%%%%%%%%%%%%%%%%%%%%%

\section{Optimizing the Vandermonde determinant \\over the unit sphere} \label{sec:analytical}
In this section we will consider the extreme points of the Vandermonde determinant on the $n$-dimensional unit sphere in $\mathbb{R}^n$. We want both to find an analytical solution and to identify some properties of the determinant that can help us to visualize it in some area around the extreme points in dimensions $n > 3$.

% intro text
% task: solve optimization problem
% task: find a method to explore the problem visually around the extreme points.

\subsection{The extreme points given by roots of a polynomial}
\label{sec:roots_by_pol}
The extreme points of the Vandermonde determinant on the unit sphere in $\mathbb{R}^n$ are known and given by Theorem \ref{thm:extreme_points_analytical} where we present a special case of Theorem 6.7.3 in 'Orthogonal polynomials' by G\'{a}bor Szeg\H{o} \cite{szego}.
We will also provide a proof that is more explicit than the one in \cite{szego}
and that exposes more of the rich symmetric properties of the Vandermonde determinant.
For the sake of convenience some properties related to the extreme points of the Vandermonde determinant defined by real vectors $\vec x_n$ will be presented before Theorem \ref{thm:extreme_points_analytical}.

%In this section we consider the Vandermonde determinant over points $\vec{x}_n \in \mathbb{R}^n$.
% , and
% without loss of generality it can be assumed that 
% \[ x_1 \leq x_2 \leq \ldots\ \leq x_n. \]

\begin{theorem}
 \label{thm:partialDerivatives}
 For any $1 \leq k \leq n$
 \begin{equation} 
  \label{eq:partialDerivatives}
  \frac{\partial v_n}{\partial x_k} = \sum_{\substack{i=1\\i\neq k}}^{n} \frac{v_n(\vec{x}_n)}{x_k-x_i}
 \end{equation}
\end{theorem}
This theorem will be proved after the introduction of the following useful lemma:
\\
\begin{lemma}
 \label{lem:partialDerivatives}
 For any $1 \leq k \leq n-1$
 \begin{equation}
  \label{eq:partialDerNotLast}
  \frac{\partial v_n}{\partial x_k} = -\frac{v_n(\vec{x}_n)}{x_n-x_k} + \left[\prod_{i=1}^{n-1} (x_n-x_i)\right] \frac{\partial v_{n-1}}{\partial x_k}
 \end{equation}
 and
 \begin{equation}
  \label{eq:partialDerLast}
  \frac{\partial v_n}{\partial x_n} = \sum_{i=1}^{n-1} \frac{v_n(\vec{x}_n)}{x_n-x_i}.
 \end{equation}
\end{lemma}
\begin{proof}
 Note that the determinant can be described recursively
 \begin{align} 
  \label{eq:recDet}
  v_n(\vec{x}_n) &= \left[\prod_{i = 1}^{n-1} (x_n-x_i)\right] \prod_{1 \leq i < j \leq n-1} (x_j-x_i) \nonumber \\
              &= \left[\prod_{i = 1}^{n-1} (x_n-x_i)\right] v_{n-1}(\vec{x}_{n-1}).
 \end{align}
  Formula \eqref{eq:partialDerNotLast} follows immediately from applying the differentiation formula for products on \eqref{eq:recDet}. Formula \eqref{eq:partialDerLast} follows from \eqref{eq:recDet}, the differentiation rule for products and that $v_{n-1}(\vec{x}_{n-1})$ is independent of $x_n$.
 \begin{align*}
  \frac{\partial v_n}{\partial x_n} =& \frac{v_{n-1}(\vec{x}_{n-1})}{x_n-x_1}\prod_{i = 1}^{n-1} (x_n-x_i) \\
                                     &+ (x_n-x_1) \frac{\partial}{\partial x_n}\left(\frac{v_{n-1}(\vec{x}_{n-1})}{x_n-x_1}\prod_{i = 1}^{n-1} (x_n-x_i) \right) \\
                                    =& \frac{v_n(\vec{x}_n)}{x_n-x_1}+\frac{v_n(\vec{x}_n)}{x_n-x_2} \\
                                     &+ (x_n-x_1)(x_n-x_2) \frac{\partial}{\partial x_n}\left( \frac{v_n(\vec{x}_n)}{(x_n-x_1)(x_n-x_2)}\right) \\
                                    =& \sum_{i=1}^{n-1} \frac{v_n(\vec{x}_n)}{x_n-x_i} + \left[\prod_{i = 1}^{n-1} (x_n-x_i)\right] \frac{\partial v_{n-1}}{\partial x_n} = \sum_{i=1}^{n-1} \frac{v_n(\vec{x}_n)}{x_n-x_i}.
 \end{align*} 
\end{proof}

\begin{proof}[Proof of Theorem \ref{thm:partialDerivatives}]
 Using lemma \ref{lem:partialDerivatives} we can see that for $k = n$ Equation \eqref{eq:partialDerivatives} follows immediately from \eqref{eq:partialDerLast}. 
 The case $1 \leq k < n$ will be proved using induction.
 Using \eqref{eq:partialDerNotLast} gives
 \begin{align*}
  \frac{\partial v_n}{\partial x_k} &= -\frac{v_n(\vec{x}_n)}{x_n-x_k} + \left[\prod_{i=1}^{n-1} (x_n-x_i)\right] \frac{\partial v_{n-1}}{\partial x_k}.
 \end{align*}
 Supposing that formula \eqref{eq:partialDerivatives} is true for $n-1$ results in
 \begin{align*}
  \frac{\partial v_n}{\partial x_k} =& -\frac{v_n(\vec{x}_n)}{x_n-x_k} + \left[\prod_{i=1}^{n-1} (x_n-x_i)\right] \sum_{\substack{i=1\\i\neq k}}^{n-1} \frac{v_{n-1}(\vec{x}_{n-1})}{x_k-x_i} \\
                                    =& \frac{v_n(\vec{x}_n)}{x_k-x_n} + \sum_{\substack{i=1\\i\neq k}}^{n-1} \frac{v_n(\vec{x}_n)}{x_k-x_i}\\
                                    =& \sum_{\substack{i=1\\i\neq k}}^{n} \frac{v_n(\vec{x}_n)}{x_k-x_i}.
 \end{align*}
 Showing that \eqref{eq:partialDerivatives} is true for $n=2$ completes the proof
 \begin{align*} 
  \frac{\partial v_2}{\partial x_1} = \frac{\partial}{\partial x_1}(x_2-x_1) = -1 &= \frac{x_2-x_1}{x_1-x_2} = \sum_{\substack{i=1\\i\neq 1}}^{2} \frac{v_2(\vec{x}_2)}{x_1-x_i} \\
  \frac{\partial v_2}{\partial x_2} = \frac{\partial}{\partial x_2}(x_2-x_1) = 1 &= \frac{x_2-x_1}{x_2-x_1} = \sum_{\substack{i=1\\i\neq 2}}^{2} \frac{v_2(\vec{x}_2)}{x_2-x_i}.
 \end{align*}
\end{proof}
\begin{theorem}
 \label{thm:hyperplane}
 The extreme points of $v_n(\vec{x}_n)$ on the unit sphere can all be found in the hyperplane defined by
 \begin{equation}
  \sum_{i=1}^{n} x_i = 0.
 \end{equation}
\end{theorem}
This theorem will be proved after the introduction of the following useful lemma:
\\
\begin{lemma}
 \label{lem:sumPartDer}
 For any $n \geq 2$ the sum of the partial derivatives of $v_n(\vec{x}_n)$ will be zero.
 \begin{equation}
  \label{eq:sumPartDer}
  \sum_{k=1}^{n} \frac{\partial v_n}{\partial x_k} = 0.
 \end{equation}
\end{lemma}
\begin{proof}
 This lemma is easily proved using lemma \ref{lem:partialDerivatives} and induction.
 Lemma \ref{eq:partialDerivatives} gives
 \begin{align*}
  \sum_{k=1}^{n} \frac{\partial v_n}{\partial x_k} =& \sum_{k=1}^{n-1} \left( -\frac{v_n(\vec{x}_n)}{x_n-x_k} + \left[\prod_{i=1}^{n-1} (x_n-x_i)\right] \frac{\partial v_{n-1}}{\partial x_k} \right) + \sum_{i=1}^{n-1} \frac{v_n(\vec{x}_n)}{x_n-x_i} \\  
                                                   =&\left[\prod_{i=1}^{n-1} (x_n-x_i)\right] \sum_{k=1}^{n-1} \frac{\partial v_{n-1}}{\partial x_k}.
 \end{align*}
 Thus if Equation \eqref{eq:sumPartDer} is true for $n-1$ it is also true for $n$. Showing that the equation holds for $n=2$ is very simple
 \[ \frac{\partial v_2}{\partial x_1} + \frac{\partial v_2}{\partial x_2} = -1 + 1 = 0. \qedhere \]
\end{proof}
\begin{proof}[Proof of Theorem \ref{thm:hyperplane}]
 Using the method of Lagrange multipliers it follows that any $\vec{x}_n$ on the unit sphere that is an extreme point of the Vandermonde determinant will also be a stationary point for the Lagrange function
 \[\Lambda_n(\vec{x}_n,\lambda) = v(\vec{x}_n) - \lambda \left(\sum_{i=1}^{n} x_i^2 - 1\right) \] 
 for some $\lambda$. Explicitly this requirement becomes
 \begin{align}
  \frac{\partial \Lambda_n}{\partial x_k} &= 0 \text{ for all } 1 \leq k \leq n \label{eq:lagrangeStatPt}, \\
  \frac{\partial \Lambda_n}{\partial \lambda} &= 0 \label{eq:lagrangeConst}.
 \end{align}
 Equation \eqref{eq:lagrangeConst} corresponds to the restriction to the unit sphere and is therefore immediately fulfilled. Since all the partial derivatives of the Lagrange function should be equal to zero it is obvious that the sum of the partial derivatives will also be equal to zero. Combining this with lemma \ref{lem:sumPartDer} gives
 \begin{equation}
  \label{eq:lagrangeSumCond}
  \sum_{k=1}^{n} \frac{\partial \Lambda_n}{\partial x_k} = \sum_{k=1}^{n}\PAR{ \frac{\partial v_n}{\partial x_k} - 2\lambda x_k }= -2\lambda \sum_{k=1}^{n} x_k = 0.
 \end{equation}
 There are two ways to fulfill condition \eqref{eq:lagrangeSumCond} either $\lambda = 0$ or $x_1 + x_2 + \ldots\ +x_n = 0$.
%  When $\lambda = 0$ Equation \eqref{eq:lagrangeStatPt} is only true when $x_i = x_j$ for some $i \neq j$ since Theorem \ref{thm:partialDerivatives} gives
%  \begin{equation}
%   \label{eq:sumLambdaZero}
%   \frac{\partial \Lambda_n}{\partial x_n} = \sum_{k=1}^{n-1} \frac{v_n(\vec{x}_n)}{x_n-x_k}.
%  \end{equation}
%  Since the elements of $\vec{x}_n$ are assumed to be ordered, $x_1 \leq x_2 \leq \ldots \leq x_n$, the sum in expression \eqref{eq:sumLambdaZero} have only non-negative terms. A term is only zero when $x_i = x_j$ for some $i \neq j$ such that $x_j-x_i$ appear in the term. The entire sum can only be zero if all the terms are equal to zero. If $x_i = x_j$ for some $i \neq j$ then the Vandermonde determinant is zero regardless of all the other elements of $\vec{x}_n$ and therefore $\lambda = 0$ can not correspond to an extreme point.
When $\lambda = 0$ Equation \eqref{eq:lagrangeStatPt} reduces to
\begin{align*}
\frac{\partial v_n}{\partial x_k}=0 \text{ for all } 1 \leq k \leq n,
\end{align*}
and by Equation \eqref{eq:vn_scalingx} this can only be true if $v_n(\vec x_n)=0$, which is of no interest to us, and so
all extreme points must lie in the hyperplane $x_1+x_2+ \ldots\ +x_n = 0$.

\end{proof}

\begin{theorem}
 \label{thm:extreme_points_analytical}
 A point on the unit sphere in $\mathbb{R}^n$, $\vec{x}_n = (x_1, x_2, \ldots\, x_n)$, is an extreme point of the Vandermonde determinant if and only if all $x_i$, $i \in \{ 1,2,\ldots\,n \}$, are distinct roots of the rescaled Hermite polynomial 
 \begin{equation}
  P_n(x) = \left(2n(n-1)\right)^{-\frac{n}{2}}H_n\left(\sqrt{\frac{n(n-1)}{2}}x\right). \label{eq:extreme_point_poly}
 \end{equation}
\end{theorem}
\vspace{1em}
\begin{remark}
\label{rem:n_factorial}
 Note that if $\vec{x}_n = (x_1, x_2, \ldots\, x_n)$ is an extreme points of the Vandermonde determinant then any other point whose coordinates are a permutation of the coordinates of $\vec{x}_n$ is also an extreme point. This follows since the determinant function, by definition, is alternating with respect to the columns of the matrix and the $x_i$s defines the columns of the Vandermonde matrix. Thus any permutation of the $x_i$s will give the same value for $\abs{v_n(\vec x_n)}$. Since there are $n!$ permutations there will be at least $n!$ extreme points. The roots of the polynomial \eqref{eq:extreme_point_poly} defines the set of $x_i$s fully and thus there are exactly $n!$ extreme points, $n!/2$ positive and $n!/2$ negative.
\end{remark}
\vspace{1em}
\begin{remark}
\label{rem:r_symmetric}
 All terms in $P_n(x)$ are of even order if $n$ is even and of odd order when $n$ is odd. This means that the roots of $P_n(x)$ will be symmetrical
 in that they come in pairs $x_i,-x_i$.
\end{remark}
\begin{proof}[Proof of Theorem \ref{thm:extreme_points_analytical}]
  By the method of Lagrange multipliers condition \eqref{eq:lagrangeStatPt} must be fulfilled for any extreme point. If $\vec{x}_n$ is a fixed extreme point so that
 \[ v_n(\vec{x}_n) = v_{max}, \]
 then \eqref{eq:lagrangeStatPt} can be written explicitly, using \eqref{eq:partialDerivatives}, as
 \[ \frac{\partial \Lambda_n}{\partial x_k} = \sum_{\substack{i=1 \\ i \neq k}}^{n} \frac{v_{max}}{x_k-x_i} - 2\lambda x_k = 0 \text{ for all } 1 \leq k \leq n ,\]
 or alternatively by introducing a new multiplier $\rho$ as
 \begin{equation}
 \label{eq:poly_derivation_1}
  \sum_{\substack{i=1 \\ i \neq k}}^{n} \frac{1}{x_k-x_i} = \frac{2\lambda}{v_{max}} x_k = \frac{\rho}{n} x_k
  \text{ for all } 1 \leq k \leq n.
 \end{equation}
%  By denoting $f(x) = (x-x_1)(x-x_2)\cdots(x-x_n)$ and noting that
%  \begin{align*}
%   f'(x) &= \sum_{\substack{i=1 \\ x \neq x_i}}^{n} \frac{f(x)}{x-x_i} \\
%   f''(x) &= \sum_{\substack{i=1 \\ x \neq x_i}}^{n} \sum_{\substack{j=1 \\ i \neq j \\ x \neq x_j}}^{n} \frac{f(x)}{(x-x_i)(x-x_j)}
%  \end{align*}
%  {\color{red} Detta är enkelt att göra för hand men jag kan inte komma på ett enkelt sätt att skriva ner det så att det blir tydligt. Jag tycker vi bara ska låtaa det vara som det är.}
%  expression \eqref{eq:poly_derivation_1} can be given in a more elegant form
 By forming the polynomial $f(x) = (x-x_1)(x-x_2)\cdots(x-x_n)$ with the coordinates of the extreme points as its $n$ roots and noting that
 \begin{align*}
  f'(x_k)= \sum_{j=1}^n \prod_{\substack{i=1\\i\neq j}}^n (x-x_i)\bigg|_{x=x_k} 
  &=  \prod_{\substack{i=1\\i\neq k}}^n (x_k-x_i),\\
  f''(x_k)= \sum_{l=1}^n \sum_{\substack{j=1\\j\neq l}}^n \prod_{\substack{i=1\\i\neq j\\i\neq l}}^n (x-x_i)\bigg|_{x=x_k}
  &=\sum_{\substack{j=1\\j\neq k}}^n \prod_{\substack{i=1\\i\neq j\\i\neq k}}^n (x_k-x_i)
  +\sum_{\substack{l=1\\l\neq k}}^n \prod_{\substack{i=1\\i\neq l\\i\neq k}}^n (x_k-x_i)\\
  &=2 \sum_{\substack{j=1\\j\neq k}}^n \prod_{\substack{i=1\\i\neq j\\i\neq k}}^n (x_k-x_i),
 \end{align*}
 we can rewrite \eqref{eq:poly_derivation_1} as
 \begin{align*}
  \frac{1}{2} \frac{f''(x_k)}{f'(x_k)} = \frac{\rho}{n} x_k,
 \end{align*}
or
 \begin{align*}
 f''(x_k) - \frac{2\rho}{n} x_k f'(x_k)  = 0.
 \end{align*}
 And since the last equation must vanish for all $k$ we must have
 \begin{align}
 f''(x) - \frac{2\rho}{n} x f'(x)  = c f(x),  \label{eq:poly_derivation_2}
 \end{align}
 for some constant $c$.
 To find $c$ the $x^n$-terms of the right and left part of Equation \eqref{eq:poly_derivation_2} are compared to each other,
 \[ c \cdot c_{n} x^n = -\frac{2\rho}{n} x n c_n x^{n-1} = -2\rho \cdot c_n x^n  ~\Rightarrow~ c = -2\rho . \]
 Thus the following differential equation for $f(x)$ must be satisfied
 \begin{equation}
 \label{eq:diff_eq_roots_basic} 
 f''(x) - \frac{2\rho}{n} x f'(x) + 2\rho f(x) = 0.
 \end{equation}
 Choosing $x = az$ gives
 \begin{align*}
   & f''(az) - \frac{2\rho}{(n-1)} a^2 z f'(az) + 2\rho f(az) \\
  =& \frac{1}{a^2} \frac{\mathrm{d}^2 f}{\mathrm{d}z^2}(az) - \frac{2\rho}{n} a z\frac{1}{a}\frac{\mathrm{d} f}{\mathrm{d}z}(az) + 2\rho f(az) = 0.
 \end{align*}
 By setting $g(z) = f(az)$ and choosing $a = \sqrt{\frac{n}{\rho}}$ a differential equation that matches the definition for the Hermite polynomials is found:
 \begin{equation}
 \label{eq:hermite_poly}
  g''(z) - 2 z g'(z) + 2 n g(z) = 0.
 \end{equation}
 By definition the solution to \eqref{eq:hermite_poly} is $g(z) = b H_n(z)$ where $b$ is a constant.
 An exact expression for the constant $a$ can be found using Lemma \ref{lem:sum_squares} (for the sake of convenience the lemma is stated and proved after this theorem).
 We get
 \[ \sum_{i=1}^{n} x_i^2 = \sum_{i=1}^{n} a^2 z_i^2 = 1 \Rightarrow a^2 \frac{n(n-1)}{2} = 1, \]
 and so
 \[ a = \sqrt{\frac{2}{n(n-1)}}. \]
 Thus condition \eqref{eq:lagrangeStatPt} is fulfilled when $x_i$ are the roots of
 \[ P_n(x) = bH_n\left(z\right) = bH_n\left(\sqrt{\frac{n(n-1)}{2}}x\right). \]
 Choosing $b = \left(2n(n-1)\right)^{-\frac{n}{2}}$ gives $P_n(x)$ leading coefficient $1$. This can be confirmed by calculating the leading coefficient of $P(x)$ using the explicit expression for the Hermite polynomial \eqref{eq:hermite_explicit}. This completes the proof.
\end{proof}

\begin{lemma}
 \label{lem:sum_squares}
 Let $x_i$, $i = 1,2,\ldots,n$ be roots of the Hermite polynomial $H_n(x)$. Then
 \[ \sum_{i=1}^{n} x_i^2 = \frac{n(n-1)}{2}. \]
\end{lemma}

\begin{proof}
 By letting $e_k(x_1,\ldots\,x_n)$ denote the elementary symmetric polynomials $H_n(x)$ can be written as
 \begin{align*} 
 H_n(x) &= A_n(x-x_1)\cdots(x-x_n) \\
        &= A_n (x^n - e_1(x_1,\ldots,x_n)x^{n-1} + e_2(x_1,\ldots,x_n)x^{n-2} + q(x))
 \end{align*}
 where $q(x)$ is a polynomial of degree $n-3$. Noting that
 \begin{align} 
 \sum_{i=1}^{n} x_i^2 &= (x_1+\ldots+x_n)^2 - 2 \sum_{1\leq i < j \leq n} x_i x_j \nonumber \\ 
                      &= e_1(x_1,\ldots,x_n)^2 - 2 e_2(x_1,\ldots,x_n) \label{eq:sum_squares},
 \end{align}
 it is clear the the sum of the square of the roots can be described using the coefficients for $x^n$, $x^{n-1}$ and $x^{n-2}$.
 The explicit expression for $H_n(x)$ is \cite{szego}
 \begin{align} 
 H_n(x) &= n! \sum_{i=0}^{\left\lfloor\frac{n}{2}\right\rfloor} \frac{(-1)^i}{i!}\frac{(2x)^{n-2i}}{(n-2i)!} \nonumber \\
        &= 2^n x^n - 2^{n-2}n(n-1) x^{n-2} + n! \sum_{i=3}^{\left\lfloor\frac{n}{2}\right\rfloor} \frac{(-1)^i}{i!}\frac{(2x)^{n-2i}}{(n-2i)!} \label{eq:hermite_explicit}.
 \end{align}
 Comparing the coefficients in the two expressions for $H_n(x)$ gives
 \begin{align*}
  A_n &= 2^n, \\
  A_n e_1(x_1,\ldots,x_n) &= 0, \\
  A_n e_2(x_1,\ldots,x_n) &= -n(n-1) 2^{n-2}.
 \end{align*}
 Thus by \eqref{eq:sum_squares}
 \[ \sum_{i=1}^{n} x_i^2 = \frac{n(n-1)}{2}. \]
\end{proof}

\begin{theorem} \label{thm:coefficients}
 The coefficients, $a_k$, for the term $x^k$ in $P_n(x)$ given by \eqref{eq:extreme_point_poly} is given by the following relations
 \begin{align}
  a_n &= 1,~~a_{n-1} = 0,~~a_{n-2} = \frac{1}{2}, \nonumber \\
  a_k &= -\frac{(k+1)(k+2)}{n(n-1)(n-k)}a_{k+2},~~1 \leq k \leq n-3. \label{eq:poly_coeff_rec}
 \end{align}
\end{theorem}
\begin{proof}
 Equation \eqref{eq:diff_eq_roots_basic} tells us that
 \begin{equation} 
  \label{eq:pn_diff_eq}
  P_n(x)=\frac{1}{2\rho}P_n''(x)-\frac{1}{n}xP_n'(x).
 \end{equation}
 That $a_n = 1$ follows from the definition of $P_n$ and $a_{n-1} = 0$ follows from the Hermite polynomials only having terms of odd powers when $n$ is odd and even powers when $n$ is even. That $a_{n-2} = \frac{1}{2}$ can be easily shown using the definition of $P_n$ and the explicit formula for the Hermite polynomials \eqref{eq:hermite_explicit}.\\
 The value of the $\rho$ can be found by comparing the $x^{n-2}$ terms in \eqref{eq:pn_diff_eq}
 \[ a_{n-2} = \frac{1}{2\rho} n(n-1) a_{n}+\frac{1}{n}(n-2)a_{n-2}. \]
 From this follows
 \[ \frac{1}{2 \rho} = \frac{-1}{n^2(n-1)}. \]
 Comparing the $x^{n-l}$ terms in \eqref{eq:pn_diff_eq} gives the following relation
 \[ a_{n-l} = \frac{1}{2\rho} (n-l+2)(n-l) a_{n-l+2} + (n-l)a_{a-l}\frac{1}{n} \]
 which is equivalent to
 \[ a_{n-l} = a_{n-l+2} \frac{-(n-l+2)(n-l+1)}{l n^2(n-1)}. \]
 Letting $k = n-l$ gives \eqref{eq:poly_coeff_rec}. 
\end{proof}
% We use Lagrangian multipliers to arrive at the hyperplane theorem.
% We present results for the derivative of the Vandermonde determinant.\\
% We present results for the derivative of Pn.
% Prove differential equation, coefficient formula, and relation to Hn.\\

% Show uniqueness of global optimum. (note that this implies ONLY symmetrical points since we must have symmetrical extreme points)\\
% We present results for the derivative of the Vandermonde determinant assuming symmetric points.
% We prove that there are symmetrical extreme points.

\subsection{Extreme points of the Vandermonde determinant on the three dimensional unit sphere} \label{sec:analytical_v3}

It is fairly simple to describe $v_3(\vec{x}_3)$ on the circle, $S^2$, that is formed by the intersection of the unit sphere and the plane $x_1+x_2+x_3=0$. Using Rodrigues' rotation formula to rotate a point, $\vec{x}$, around the axis $\frac{1}{\sqrt{3}}(1,1,1)$ with the angle $\theta$ will give the rotation matrix
\[ \setlength{\arraycolsep}{1.5pt}
   R_\theta = \frac{1}{3}\begin{bmatrix} 
                   2\cos (\theta) + 1 & 1 - \cos (\theta) - \sqrt{3}\sin(\theta) & 1 - \cos (\theta) + \sqrt{3}\sin(\theta) \\
                   1 - \cos (\theta) + \sqrt{3}\sin(\theta) & 2\cos (\theta) + 1 & 1 - \cos (\theta) - \sqrt{3}\sin(\theta) \\ 
                   1 - \cos (\theta) - \sqrt{3}\sin(\theta) & 1 - \cos (\theta) + \sqrt{3}\sin(\theta) & 2\cos (\theta) + 1 
                  \end{bmatrix}\hspace{-3pt}. \]
                  
A point which already lies on $S^2$ can then be rotated to any other point on $S^2$ by letting $R_\theta$ act on the point. Choosing the point $\vec{x} = \frac{1}{\sqrt{2}}\left(-1,0,1\right)$ gives the Vandermonde determinant a convenient form on the circle since:
\[ R_\theta \vec{x} = \frac{1}{\sqrt{6}}\begin{bmatrix}
                                          -\sqrt{3}\cos(\theta)-\sin(\theta) \\
                                          -2\sin(\theta) \\
                                          \sqrt{3}\cos(\theta)+\sin(\theta)
                                         \end{bmatrix}, \]                
which gives
\begin{align*} 
 v_3(R_\theta \vec{x}) = 2&\left(\sqrt{3}\cos(\theta)+\sin(\theta)\right) \\
                           &\left(\sqrt{3}\cos(\theta)+\sin(\theta)+2\sin(\theta)\right) \\
                           &\left(-2\sin(\theta)+\sqrt{3}\cos(\theta)+\sin(\theta)\right) \\
                        =  &\frac{1}{\sqrt{2}}\left(4\cos(\theta)^3-3\cos(\theta)\right) \\
                        =  &\frac{1}{\sqrt{2}}\cos(3\theta).
\end{align*}
Note that the final equality follows from $\cos(n\theta) = T_n(\cos(\theta))$ where $T_n$ is the $n$th Chebyshev polynomial of the first kind. From formula \eqref{eq:extreme_point_poly} if follows that $P_3(x) = T_3(x)$ but for higher dimensions the relationship between the Chebyshev polynomials and $P_n$ is not as simple.\\
Finding the maximum points for $v_3(\vec{x}_3)$ on this form is simple. The Vandermonde determinant will be maximal when $3\theta = 2n\pi$ where $n$ is some integer. This gives three local maximums corresponding to $\theta_1 = 0$, $\theta_2 = \frac{2\pi}{3}$ and $\theta_3 = \frac{4\pi}{3}$. These points correspond to cyclic permutation of the coordinates of $\vec{x} = \frac{1}{\sqrt{2}}\left(-1,0,1\right)$. Analogously the minimas for $v_3(\vec{x}_3)$ can be shown to be a transposition followed by cyclic permutation of the coordinates of $\vec{x}$. Thus any permutation of the coordinates of $\vec{x}$ correspond to a local extreme point just like it was stated on page \pageref{rem:permutations3D}.
% solution by trigonometry.

\subsection{A transformation of the optimization problem} \label{sec:transformation}

\begin{lemma} \label{lem:vndotgradient}
For any $n\geq 2$ the dot product between the gradient of $v_n(\vec{x}_n)$ and $\vec x_n$ is proportional to $v_n(\vec{x}_n)$.
More precisely,
\begin{align}
 \nabla v_n^T \vec{x}_n
 =
 \sum_{k=1}^n x_k\frac{\partial v_n}{\partial x_k}
 =
 \frac{n(n-1)}{2}v_n(\vec{x}_n).
\end{align}
\end{lemma}
\begin{proof}
 Using Theorem \ref{thm:partialDerivatives} we have
\begin{align*}
  \sum_{k=1}^n x_k\frac{\partial v_n}{\partial x_k}
 =\sum_{k=1}^n x_k\sum_{i\neq k} \frac{v_n(\vec{x}_n)}{x_k-x_i}
 =v_n(\vec{x}_n) \sum_{k=1}^n \sum_{i\neq k} \frac{x_k}{x_k-x_i}.
\end{align*}
Now, for each distinct pair of indices $k=a,i=b$ in the last double sum we have that the indices $k=b,i=a$ also appear.
And so we continue
\begin{align*}
 \sum_{k=1}^n x_k\frac{\partial v_n}{\partial x_k}& = v_n(\vec{x}_n) \sum_{1\leq k<i\leq n}^n\PAR{
 \frac{x_k}{x_k-x_i}
 +
 \frac{x_i}{x_i-x_k}
 }
 \\
 &= v_n(\vec{x}_n) \sum_{1\leq k<i\leq n}^n 1 = \frac{n(n-1)}{2} v_n(\vec{x}_n),
\end{align*}
which proves the lemma.
\end{proof}

Consider the premise that an objective function $f(\vec x)$ is optimized on the points satisfying an equality constraint $g(\vec x)=0$ when its gradient is linearly dependent with the gradient of the constraint function.
In Lagrange multipliers this is expressed as
\begin{align*}
\nabla f(\vec x)=\lambda \nabla g(\vec x),
\end{align*}
where $\lambda$ is some scalar constant.
We can also express this using a dot product
\begin{align*}
\nabla f(\vec x)\cdot \nabla g(\vec x)=|\nabla f(\vec x)||\nabla g(\vec x)|\cos \theta.
\end{align*}
We are interested in the case where both $\nabla f$ and $\nabla g$ are non-zero
and so for linear dependence we require $\cos \theta=\pm 1$.
By squaring we then have
\begin{align*}
\left( \nabla f(\vec x)\cdot \nabla g(\vec x) \right)^2 =
\left(\nabla f(\vec x)\cdot \nabla f(\vec x)\right)
\left(\nabla g(\vec x)\cdot \nabla g(\vec x)\right),
\end{align*}
which can also be expressed
\begin{align} \label{eq:lineardependence}
\left( \sum_{i=1}^n \frac{\partial f}{\partial x_i} \frac{\partial g}{\partial
x_i} \right)^2=
\left( \sum_{i=1}^n \left( \frac{\partial f}{\partial x_i} \right)^2 \right)
\left( \sum_{i=1}^n \left( \frac{\partial g}{\partial x_i} \right)^2 \right).
\end{align}

\begin{theorem} \label{thm:transformation}
The problem of finding the vectors $\vec x_n$ that maximize the absolute value of the Vandermonde determinant over the unit sphere:
\begin{align}
\begin{cases}
\displaystyle\max_{\vec x_n} \prod_{i<j} | x_j-x_i |,\\
\displaystyle\text{s.t. } \sum_{i=1}^{n} x_i^2=1,
\end{cases},
\end{align}
has exactly the same solution set as the related problem
\begin{align}
\begin{cases} \label{eq:equi}
\displaystyle
\sum_{i<j} \frac{1}{(x_j-x_i)^2} = \frac 1 2 \left(  \frac{n(n-1)}{2} 
\right)^2,\\
\displaystyle
\sum_{i=1}^{n} x_i^2=1.
\end{cases}
\end{align}
\end{theorem}

\begin{proof}

By applying Equation \ref{eq:lineardependence} to the problem of optimizing the Vandermonde determinant $v_n$ over the unit
sphere we get.
\begin{align}
\left( \sum_{i=1}^n \frac{\partial v_n}{\partial x_i} \frac{\partial
\sum x_i^2}{\partial x_i} \right)^2 &=
\left( \sum_{i=1}^n \left( \frac{\partial v_n}{\partial x_i} \right)^2 \right)
\left( \sum_{i=1}^n \left( \frac{\partial \sum x_i^2}{\partial x_i} \right)^2
\right),\nonumber\\
\left( \sum_{i=1}^n 2x_i\frac{\partial v_n}{\partial x_i} \right)^2&=
\left( \sum_{i=1}^n \left( \frac{\partial v_n}{\partial x_i} \right)^2 \right)
\left( \sum_{i=1}^n \left( 2x_i \right)^2
\right),\nonumber\\
\left(
	\sum_{i=1}^n x_i\frac{\partial v_n}{\partial x_i}
\right)^2
&=
\sum_{i=1}^n \left( \frac{\partial v_n}{\partial x_i} \right)^2. \label{eq:withderivatives}
\end{align}

By applying Lemma \ref{lem:vndotgradient} the left part of Equation (\ref{eq:withderivatives}) can be written
\begin{align*}
v_n(\vec{x}_n)^2\left(
 \frac{n(n-1)}{2}
\right)^2.
\end{align*}

The right part of Equation (\ref{eq:withderivatives}) can be rewritten as
\begin{align*}
\sum_{i=1}^n \left( \frac{\partial v_n}{\partial x_i} \right)^2
=
\sum_{i=1}^n \left( \sum_{\substack{k=1\\k\neq i}}^n \frac{v_n(\vec{x}_n)}{x_i-x_k} \right)^2
=
v_n(\vec{x}_n)^2
\sum_{i=1}^n \left( \sum_{\substack{k=1\\k\neq i}}^n \frac{1}{x_i-x_k}
\right)^2,
\end{align*}
and by expanding the square we continue
\begin{align*}
\sum_{i=1}^n \left( \frac{\partial v_n}{\partial x_i} \right)^2 =
v_n(\vec{x}_n)^2
\sum_{i=1}^n \left(
\sum_{\substack{k=1\\k\neq i}}^n \frac{1}{(x_i-x_k)^2}
+
\sum_{\substack{k=1\\k\neq i}}^n
\sum_{\substack{j=1\\j\neq i\\j\neq k}}^n
\frac{1}{(x_i-x_k)} \frac{1}{(x_i-x_j)}
\right)
\\
=
v_n(\vec{x}_n)^2
\sum_{\substack{k\neq i}} \frac{1}{(x_i-x_k)^2}
+
v_n(\vec{x}_n)^2
\sum_{i=1}^n
\sum_{\substack{k=1\\k\neq i}}^n
\sum_{\substack{j=1\\j\neq i\\j\neq k}}^n
\frac{1}{(x_i-x_k)} \frac{1}{(x_i-x_j)}.
\end{align*}
We recognize that the triple sum runs over all distinct $i,k,j$ and so we
can write them as one sum by expanding permutations:
\begin{align*}
\sum_{i=1}^n \left( \frac{\partial v_n}{\partial x_i} \right)^2 =&
~v_n(\vec{x}_n)^2
\sum_{\substack{k\neq i}} \frac{1}{(x_i-x_k)^2}
\\
&+
v_n(\vec{x}_n)^2
\sum_{i<j<k}
\left(\frac{1}{(x_i-x_k)(x_i-x_j)}
+\frac{1}{(x_i-x_j)(x_i-x_k)}\right.
\\
&~\hspace{2.9cm}+\frac{1}{(x_j-x_k)(x_j-x_i)}
+\frac{1}{(x_j-x_i)(x_j-x_k)}
\\
&\left.~\hspace{2.8cm}+\frac{1}{(x_k-x_i)(x_k-x_j)}
+\frac{1}{(x_k-x_j)(x_k-x_i)}
\right)
\end{align*}
\begin{align*}
=&
~v_n(\vec{x}_n)^2
\sum_{\substack{k\neq i}} \frac{1}{(x_i-x_k)^2}
+
2v_n(\vec{x}_n)^2
\sum_{i<j<k}
\frac{(x_k-x_j)+(x_i-x_k)+(x_j-x_i)}{(x_i-x_j)(x_j-x_k)(x_k-x_i)}
\\
=&
~v_n(\vec{x}_n)^2
\sum_{\substack{k\neq i}} \frac{1}{(x_i-x_k)^2}
=
2v_n(\vec{x}_n)^2
\sum_{i<j} \frac{1}{(x_j-x_i)^2}.
\end{align*}

We continue by joining the simplified left and right part of equation
(\ref{eq:withderivatives}).
\begin{align*}
v_n(\vec{x}_n)^2\left(
 \frac{n(n-1)}{2}
\right)^2
=
2v_n(\vec{x}_n)^2
\sum_{i<j} \frac{1}{(x_j-x_i)^2},
\end{align*}
and the result follows as
\begin{align} \label{eq:fin}
\sum_{i<j} \frac{1}{(x_j-x_i)^2}
=
\frac 1 2 \left(
 \frac{n(n-1)}{2}
\right)^2.
\end{align}
This captures the linear dependence requirement of the problem, what remains is to require the solutions to lie on the unit sphere.
\begin{equation*}
\sum_{i=1}^n x_i^2=1. \qedhere
\end{equation*}

\end{proof}

% We use the scalar product to arrive at the transformation of the optimization problem into a sum of inverse squares.
% We discuss how transformations may be formed and if sum 1/(xj-xi)^n works for any other n than 2 or perhaps 1.
% We state the open problem of finding a path directly from the transformed problem to the polynomial relations.

\subsection{Further visual exploration} \label{sec:visualnd}

Visualization of the determinant $v_3(\vec x_3)$ on the unit sphere is straightforward, as well as visualizations for $g_3(\vec x_3,\vec a)$ for different $\vec a$.
All points on the sphere can be viewed directly by a contour map.
In higher dimensions we need to reduce the set of visualized points somehow.
In this section we provide visualizations for $v_4,\cdots,v_7$ by using symmetric properties of the Vandermonde determinant.

\subsubsection{Four dimensions}
%By Theorem \ref{thm:extreme_points_analytical} or \ref{thm:coefficients} we see that the polynomials providing the coordinates of the extreme points have all even or all odd powers.
%From this it is easy to see that all coordinates of the extreme points must come in pairs $x_i,-x_i$.
%Furthermore, by Theorem \ref{thm:hyperplane} we know that the extreme points of $v_4(\vec{x}_4)$ on the sphere all lie in the hyperplane $x_1+x_2+x_3+x_4=0$.

By Theorem \ref{thm:hyperplane} we know that the extreme points of $v_4(\vec{x}_4)$ on the sphere all lie in the hyperplane $x_1+x_2+x_3+x_4=0$. This hyperplane forms a three-dimensional unit sphere and can then be easily visualized.

This can be realized using the transformation
\begin{align} \label{eq:xt_transform4}
 \vec x=\MAT{
 -1 & -1 &  0 \\
 -1 &  1 &  0 \\
  1 &  0 & -1 \\
  1 &  0 &  1
 }
 \MAT{
  1/\sqrt{4} & 0 & 0  \\
  0 & 1/\sqrt{2} & 0  \\
  0 & 0 & 1/\sqrt{2}  \\
 }
 \vec t.
\end{align}

%We use this to visualize $v_4(\vec{x}_4)$ by selecting a subspace of $\R^4$ that contains all symmetrical points $(x_1,x_2,-x_2,-x_1)$ on the sphere.
%We choose the transformation
%\begin{align} \label{eq:xt_transform4}
% \vec x=\MAT{
% -1 & 0  & 1\\
%0 & -1  & 1\\
% 0 & 1 & 1\\
% 1 & 0 & 1
% }
% \MAT{
%  1/\sqrt{2} & 0 & 0  \\
%  0 & 1/\sqrt{2} & 0  \\
%  0 & 0 & 1/\sqrt{4}  \\
% }
% \vec t.
%\end{align}

\TWOFIG{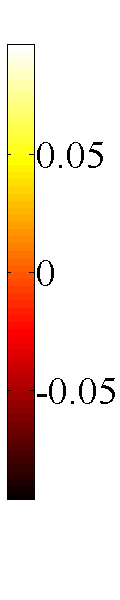}{Plot of $v_4(\vec{x}_4)$ over points on the unit sphere.}
                    {Plot with $\vec{t}$-basis given by \eqref{eq:xt_transform4}.}
                    {Plot with $\theta$ and $\phi$ given by \eqref{eq:t_spherical}.}

%\TWOFIG{g4_012plots}{Plot of $v_4(\vec{x}_4)$ over points on the unit sphere.}
%                    {Plot with respect to the $\vec{t}$-basis, see \eqref{eq:xt_transform4}.}
%                    {Plot with respect to parametrization \eqref{eq:t_spherical}.}

The results of plotting the $v_4(\vec{x}_4)$ after performing this transformation can be seen in Figure \ref{fig:g4_allplots}. All $24 = 4!$ extreme points are clearly visible.

From Figure \ref{fig:g4_allplots} we see that whenever we have a local maxima we have a local maxima at the opposite side of the sphere as well, and the same for minima.
This it due to the occurrence of the exponents in the rows of $V_n$.
From Equation \eqref{eq:vn_scalingx} we have
\begin{align*}
 v_n((-1)\vec x_n)=(-1)^{\displaystyle\frac{n(n-1)}{2}} v_n(\vec x_n),
\end{align*}
and so opposite points are both maxima or both minima if $n=4k$ or $n=4k+1$ for some $k \in \mathbb{Z}^+$ and opposite points are of different types if $n=4k-2$ or $n=4k-1$ for some $k \in \mathbb{Z}^+$.

By Theorem \ref{thm:extreme_points_analytical} the extreme points on the unit sphere for $v_4(\vec{x}_4)$ is described by the roots of this polynomial
\[ P_4(x) = x^4	-\frac{1}{2} x^2 + \frac{1}{48}. \]
The roots of $P_4(x)$ are:
\begin{align*}
x_{41} =-\frac{1}{2} \sqrt{1+\sqrt{\frac{2}{3}}},&~~x_{42} =-\frac{1}{2} \sqrt{1-\sqrt{\frac{2}{3}}},\\
x_{43} =\frac{1}{2} \sqrt{1-\sqrt{\frac{2}{3}}},&~~x_{44} =\frac{1}{2} \sqrt{1+\sqrt{\frac{2}{3}}}.
\end{align*}

\subsubsection{Five dimensions}

By Theorem \ref{thm:extreme_points_analytical} or \ref{thm:coefficients} we see that the polynomials providing the coordinates of the extreme points have all even or all odd powers.
From this it is easy to see that all coordinates of the extreme points must come in pairs $x_i,-x_i$.
Furthermore, by Theorem \ref{thm:hyperplane} we know that the extreme points of $v_5(\vec{x}_5)$ on the sphere all lie in the hyperplane $x_1+x_2+x_3+x_4+x_ 5=0$.

We use this to visualize $v_5(\vec{x}_5)$ by selecting a subspace of $\R^5$ that contains all symmetrical points $(x_1,x_2,0,-x_2,-x_1)$ on the sphere.
%As for $v_4(\vec{x}_4)$ we use symmetry to visualize $v_5(\vec{x}_5)$ by selecting a subspace of $\R^5$ that contains all points $(x_1,x_2,0,-x_2,-x_1)$ on the sphere.

The coordinates in Figure \ref{fig:g5_012plots_a} are related to $\vec{x}_5$ by
\begin{align} \label{eq:xt_transform5}
  \vec x_5=
  \MAT{-1 & 0 & 1\\0&-1&1\\0&0&1\\0&1&1\\1&0&1}
  \MAT{1/\sqrt{2}&0&0\\0&1/\sqrt{2}&0\\0&0&1/\sqrt{5}}
  \vec t .
\end{align}

\TWOFIG{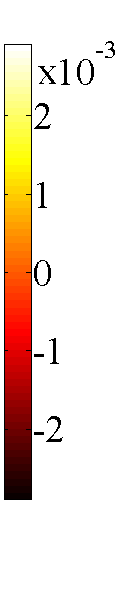}{Plot of $v_5(\vec{x}_5)$ over points on the unit sphere.}
                    {Plot with $\vec{t}$-basis given by \eqref{eq:xt_transform5}.}
                    {Plot with $\theta$ and $\phi$ given by \eqref{eq:t_spherical}.}

The result, see Figure \ref{fig:g5_012plots}, is a visualization of a subspace containing 8 of the 120 extreme points. Note that to fulfill the condition that the coordinates should be symmetrically distributed pairs can be fulfilled in two other subspaces with points that can be describes in the following ways: $(x_1,x_2,0,-x_1,-x_2)$ and $(x_2,-x_2,0,x_1,-x_1)$. This means that a transformation similar to \eqref{eq:xt_transform5} can be used to describe $3 \cdot 8 = 24$ different extreme points.

The transformation \eqref{eq:xt_transform5} corresponds to choosing $x_3 = 0$. Choosing another coordinate to be zero will give a different subspace of $\mathbb{R}^5$ which behaves identically to the visualized one. This multiplies the number of extreme points by five to the expected $5 \cdot 4! = 120$.
\\                 
By theorem \ref{thm:extreme_points_analytical} the extreme points on the unit sphere for $v_5(\vec{x}_5)$ is described by the roots of this polynomial
\[ P_5(x)=x^5 - \frac{1}{2} x^3 + \frac{3}{80} x. \]
The roots of $P_5(x)$ are:
\begin{align*}
x_{51} &=-x_{55},~~x_{52} = -x_{54},~~x_{53} = 0, \\
x_{54} &=\frac{1}{2} \sqrt{1-\sqrt{\frac{2}{5}}}, ~~ x_{55} =\frac{1}{2} \sqrt{1+\sqrt{\frac{2}{5}}}.\\
\end{align*}

\subsubsection{Six dimensions}

As for $v_5(\vec{x}_5)$ we use symmetry to visualize $v_6(\vec{x}_6)$.
We select a subspace of $\R^6$ that contains all symmetrical points $(x_1,x_2,x_3,-x_3,-x_2,-x_1)$ on the sphere.

The coordinates in Figure \ref{fig:g6_012plots_a} are related to $\vec{x}_6$ by
\begin{align} \label{eq:xt_transform6}
  \vec x_6=
  \MAT{
  -1 & 0 & 0\\
  0&-1&0\\
  0&0&-1\\
  0&0&1\\
  0&1&0\\
  1&0&0
  }
  \MAT{1/\sqrt{2}&0&0\\0&1/\sqrt{2}&0\\0&0&1/\sqrt{2}}
  \vec t.
\end{align}

\TWOFIG{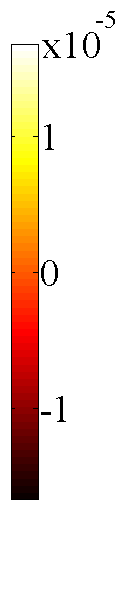}{Plot of $v_6(\vec{x}_6)$ over points on the unit sphere.}
                    {Plot with $\vec{t}$-basis given by \eqref{eq:xt_transform6}.}
                    {Plot with $\theta$ and $\phi$ given by \eqref{eq:t_spherical}.}

In Figure \ref{fig:g6_012plots} there are $48$ visible extreme points. The remaining extreme points can be found using arguments analogous the five-dimensional case.

By Theorem \ref{thm:extreme_points_analytical} the extreme points on the unit sphere for $v_6(\vec{x}_6)$ is described by the roots of this polynomial
\[ P_6(x) = x^6	- \frac{1}{2} x^4 + \frac{1}{20} x^2 - \frac{1}{1800}. \]

The roots of $P_6(x)$ are:
\begin{align}
 x_{61} =& -x_{66},~~x_{62} = -x_{65},~~x_{63} = -x_{64}, \nonumber \\
 x_{64} =& \frac{(-1)^{\frac{3}{4}}}{2\sqrt{15}}\left(10 i -\sqrt[3]{10} \left(z_6 w_6^\frac{1}{3}+ \overline{z}_6 \overline{w}_6^\frac{1}{3}\right)\right)^\frac{1}{2} \nonumber \\
        =& \frac{1}{2\sqrt{15}}\sqrt{10 - 2 \sqrt{10} \left(\sqrt{3} l_6 - k_6\right)}, \\
 x_{65} =& \frac{(-1)^{\frac{1}{4}}}{2\sqrt{15}}\left(-10 i -\sqrt[3]{10} \left(\overline{z}_6 w_6^\frac{1}{3}+z_6\overline{w}_6^\frac{1}{3}\right)\right)^\frac{1}{2} \nonumber \\
        =& \frac{1}{2\sqrt{15}}\sqrt{10 - 2 \sqrt{10} \left(\sqrt{3} l_6 + k_6\right)}, \\
 x_{66} =& \left(\frac{1}{30} \left(\sqrt[3]{10}\left(w_6^\frac{1}{3}+ \overline{w}_6^\frac{1}{3}\right)+5\right)\right)^\frac{1}{2} \nonumber \\
        =& \sqrt{\frac{1}{30} \left(2\sqrt{10}\cdot k_6 + 5\right)}, \\
    z_6 =& \sqrt{3}+i,~w_6 = 2+i \sqrt{6} \nonumber \\
    k_6 =& \cos\left(\frac{1}{3}\arctan\left(\sqrt{\frac{3}{2}}\right)\right),~
    l_6 = \sin\left(\frac{1}{3}\arctan\left(\sqrt{\frac{3}{2}}\right)\right). \nonumber
\end{align}

\subsubsection{Seven dimensions}

As for $v_6(\vec{x}_6)$ we use symmetry to visualize $v_7(\vec{x}_7)$.
We select a subspace of $\R^7$ that contains all symmetrical points $(x_1,x_2,x_3,0,-x_3,-x_2,-x_1)$ on the sphere.

The coordinates in Figure \ref{fig:g7_012plots_a} are related to $\vec{x}_7$ by
\begin{align} \label{eq:xt_transform7}
  \vec x_7=
  \MAT{
  -1 & 0 & 0\\
  0&-1&0\\
  0&0&-1\\
  0&0&0\\
  0&0&1\\
  0&1&0\\
  1&0&0
  }
  \MAT{1/\sqrt{2}&0&0\\0&1/\sqrt{2}&0\\0&0&1/\sqrt{2}}
  \vec t.
\end{align}

\TWOFIG{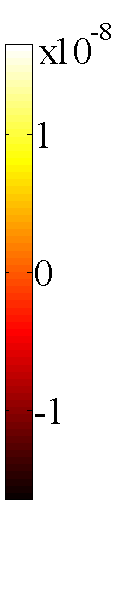}{Plot of $v_7(\vec{x}_7)$ over points on the unit sphere.}
                    {Plot with $\vec{t}$-basis given by \eqref{eq:xt_transform7}.}
                    {Plot with $\theta$ and $\phi$ given by \eqref{eq:t_spherical}.}

In Figure \ref{fig:g7_012plots} $48$ extreme points are visible just like it was for the six-dimensional case. This is expected since the transformation corresponds to choosing $x_4 = 0$ which restricts us to a six-dimensional subspace of $\mathbb{R}^7$ which can then be visualized in the same way as the six-dimensional case. The remaining extreme points can be found using arguments analogous the five-dimensional case.

By theorem \ref{thm:extreme_points_analytical} the extreme points on the unit sphere for $v_4$ is described by the roots of this polynomial
\[ P_7(x) = x^7	-\frac{1}{2} x^5 + \frac{5}{84} x^3	- \frac{5}{3528} x. \]

The roots of $P_7(x)$ are:
\begin{align}
 x_{71} =& -x_{77},~~x_{72} = -x_{76},~~x_{73} = -x_{75},~~x_{74} = 0, \nonumber \\
 x_{75} =& \frac{(-1)^{\frac{3}{4}}}{2\sqrt{21}}\left(14 i -\sqrt[3]{14} \left(z_6 w_6^\frac{1}{3}+\overline{z}_6\overline{w}_6^\frac{1}{3}\right)\right)^\frac{1}{2} \nonumber \\
        =& \frac{1}{2\sqrt{21}}\sqrt{14 - 2 \sqrt{14} \left(\sqrt{3} l_6 - k_6\right)}, \\
 x_{76} =& \frac{(-1)^\frac{1}{4}}{2\sqrt{21}}\left(-14 i -\sqrt[3]{14} \left(\overline{z}_6 w_6^\frac{1}{3}+z_6 \overline{w}_6^\frac{1}{3}\right)\right)^\frac{1}{2} \nonumber \\
        =& \frac{1}{2\sqrt{21}}\sqrt{14 - 2 \sqrt{14} \left(\sqrt{3} l_7 + k_7\right)}, \\
 x_{77} =& \sqrt{\frac{1}{42}} \left(\sqrt[3]{14} \left(z_6^\frac{1}{3}+ \overline{w}_6^\frac{1}{3}\right)+5\right)^\frac{1}{2} \nonumber \\
        =& \sqrt{\frac{1}{42} \left(2\sqrt{14} k_7 + 5 \right)}, \\
    z_6 =& \sqrt{3}+i,~w_6 = 2+i \sqrt{10} \nonumber \\
    k_7 =& \cos\left(\frac{1}{3}\arctan\left(\sqrt{\frac{5}{2}}\right)\right), \nonumber \\
    l_7 =& \sin\left(\frac{1}{3}\arctan\left(\sqrt{\frac{5}{2}}\right)\right). \nonumber
\end{align}
% We list analytic polynomials and extreme points up to degree 7, match visualization.
% We visualize v4 using a basis in the hyperplane.
% We identify the symmetric structure and set up formulas for symmetric points.
% We visualize v5, v6, v7 assuming symmetry and count extreme points.
% We count the number of extreme points numerically and see that they seem to count n!.

\section{Some limit theorems involving the\\ generalized Vandermonde matrix} \label{sec:power_limit}

Let $D_k$ be the diagonal matrix
\[
D_k=\diag\PAR{\frac{1}{0!},\frac{1}{1!},\cdots,\frac{1}{(k-1)!}}.
\]

\begin{theorem} \label{thm:gvanderpower}
For any $\vec x\in \C^n$ and $\vec a\in \C^m$ with $x_j\neq 0$ for all $j$ we have
\begin{equation}
 G_{m n}(\vec x,\vec a)=\lim_{k\rightarrow \infty} V_{k m}(\vec a)^T D_k V_{k n}(\log \vec x), \label{eq:gvanderpower}
\end{equation}
where the convergence is entry-wise, $\log \vec x=(\log x_1,\cdots,\log x_n)$ and the branch of the complex logarithm $\log(\cdots)$ is fixed and defines the expression $x_j^{a_i}$ by
\[
x_j^{a_i}\coloneqq e^{a_i \log x_j}.
\]
\end{theorem}
We will prove this theorem after presenting some results for a larger class of matrices.

Generalized Vandermonde matrices is a special case of matrices on the form
\begin{equation*}
 A_{m n}(\vec x,\vec a)=\BPAR{f(x_j,a_i)}_{m n},
\end{equation*}
where $f$ is a function.
Suppose that $\vec x$ is fixed, then each entry will be a function of one variable
\begin{equation}
 A_{m n}(\vec x,\vec a)=\BPAR{f_j(a_i)}_{m n}. \label{eq:altfunc}
\end{equation}
If all these functions $f_j$ are analytic in a neighborhood of some common $a_0$ then the functions have power series expansions around $a_0$.
If we denote the power series coefficients for function $f_j$ as $c_{j*}$ then we may write
\begin{align}
A_{m n}(\vec x,\vec a)&=\BPAR{\sum_{k=0}^{\infty} c_{jk}(a_i-a_0)^k}_{m n} \nonumber
\\
&=\lim_{k\rightarrow \infty}\BPAR{(a_i-a_0)^{j-1}}_{m k} \BPAR{c_{j(i-1)}}_{k n} \nonumber
\\
&=\lim_{k\rightarrow \infty} V_{k m}(\vec a-a_0)^T\BPAR{c_{j(i-1)}}_{k n}, \label{eq:altpower}
\end{align}
where convergence holds for each entry of $A_{m n}$ and 
\[ \vec a-a_0=(a_1-a_0,\ldots,a_m-a_0). \]

\begin{proof}[Proof of Theorem \ref{thm:gvanderpower}]
With the complex logarithmic function $\log(\cdots)$ defined to lie in a fixed branch we may write generalized Vandermonde matrices as
\begin{align*}
&   G_{m n}(\vec x,\vec a)=\BPAR{f_j(a_i)}_{m n},
\end{align*}
where
\[
 f_j(a_i)=x_j^{a_i}=e^{a_i \log x_j}, \quad 1\leq j\leq n.
\]
These functions $f_j$ are analytic everywhere whenever $x_j\neq 0$.
By the power series of the exponential function we have
\[
 f_j(a_i)=e^{a_i \log x_j}=\sum_{k=0}^{\infty} \frac{(a_i \log x_j)^k}{k!}=\sum_{k=0}^{\infty} \frac{(\log x_j)^k}{k!}a_i^k,
\]
and by Equation (\ref{eq:altpower}) we get
\begin{align*}
   G_{m n}(\vec x,\vec a) & =\lim_{k\rightarrow \infty} V_{k m}(\vec a)^T\BPAR{\frac{(\log x_j)^{i-1}}{(i-1)!}}_{k n} \\
&  =\lim_{k\rightarrow \infty} V_{k m}(\vec a)^T \diag\PAR{\frac{1}{0!},\cdots,\frac{1}{(k-1)!}} V_{k n}(\log \vec x),
\end{align*}
which concludes the proof.
\end{proof}

\begin{theorem} \label{thm:generaldetfractionlimit}
If $n\geq 2$, $\vec x,\vec a\in \C^n$, $x_j\neq 0$ for all $j$ and $v_n(\vec a)\neq 0$ then
\begin{align*}
    \lim_{t\rightarrow 0} \frac{ g_n(\vec x,\vec at) }{ v_n(\vec at) }		=\PAR{\prod_{k=1}^{n}\frac{1}{(k-1)!}} \PAR{\prod_{1\leq i<j\leq n} (\log\,x_j-\log\,x_i) }.
\end{align*}
\end{theorem}
We will prove this theorem after some intermediate results.

Let $\vec i_n=(1,2,\cdots,n)$, $P_{kn}$ be the set of all vectors $\vec p\in \N^n$ such that 
\[ 1\leq p_1<p_2<\cdots<p_n\leq k \] 
and $Q_{kn}=\{\vec p\in P_{kn}:p_n=k\}$.
An $N\times N$ minor of a matrix $A\in M_{m n}$ is determined by two vectors $\vec k\in P_{mN}$ and $\vec l\in P_{nN}$ and is defined as
\[
 A\MINOR{\vec k\\\vec l}\coloneqq \det \BPAR{A_{k_il_j}}_{N N}.
\]
Using this notation the determinant of the product of two matrices $A\in M_{n k}$ and $B\in M_{k n}$ can be written using the Cauchy-Binet formula \cite[p. 18]{Serre} as
\begin{equation}
\det(AB)=\sum_{\vec p\in P_{kn}} A\MINOR{\vec i_n\\\vec p} \cdot B\MINOR{\vec p\\\vec i_n}. \label{eq:cauchybinet}
\end{equation}

\begin{lemma} \label{lem:generaldetlimit}
If $\vec x,\vec a\in \C^n$ and $x_j\neq 0$ for all $j$ then we can write the determinant of generalized Vandermonde matrices as
\begin{align*}
&    g_n(\vec x,\vec a)=
\\&  \sum_{k=n}^{\infty} \sum_{\vec q\in Q_{kn}}
		V_{k n}(\vec a)^T \MINOR{\vec i_n\\\vec q}
		\cdot D_k \MINOR{\vec q\\\vec q}
		\cdot V_{k n}(\log \vec x) \MINOR{\vec q\\\vec i_n}.
\end{align*}
\end{lemma}
\begin{proof}
By Equation (\ref{eq:gvanderpower}), the continuity of the determinant function, the associativity of matrix multiplication, and Equation (\ref{eq:cauchybinet}) we get
\begin{align*}
g_n(\vec x,\vec a)&=\det\PAR{ \lim_{k\rightarrow \infty} V_{k n}(\vec a)^T D_k V_{k n}(\log \vec x) }
\\
&=\lim_{k\rightarrow \infty} \det\PAR{V_{k n}(\vec a)^T D_k V_{k n}(\log \vec x) }
\\
&=\lim_{k\rightarrow \infty} \det\PAR{\PAR{V_{k n}(\vec a)^T D_k} V_{k n}(\log \vec x) }
\\
&=\lim_{k\rightarrow \infty} \sum_{\vec p\in P_{kn}} \PAR{V_{k n}(\vec a)^T D_k}\MINOR{\vec i_n\\\vec p} \cdot V_{k n}(\log \vec x)\MINOR{\vec p\\\vec i_n}.
\end{align*}
We recognizing that $D_k$ is a diagonal matrix that scales the columns of $V_{k n}(\vec a)^T$:
\[
\PAR{V_{k n}(\vec a)^T D_k}(i,j)=\PAR{V_{k n}(\vec a)^T}(i,j) D_k(j,j),
\]
and so
\begin{align*}
\PAR{V_{k n}(\vec a)^T D_k}\MINOR{\vec i_n\\\vec p} & = \PAR{V_{k n}(\vec a)^T}\MINOR{\vec i_n\\\vec p} \prod_{l=1}^{n} D_k(p_l,p_l)
\\
& = \PAR{V_{k n}(\vec a)^T}\MINOR{\vec i_n\\\vec p} \cdot D_k\MINOR{\vec p\\\vec p},
\end{align*}
that is
\begin{align*}
 g_n(\vec x,\vec a) =\lim_{k\rightarrow \infty} \sum_{\vec p\in P_{kn}} V_{k n}(\vec a)^T\MINOR{\vec i_n\\\vec p} \cdot D_k\MINOR{\vec p\\\vec p} \cdot V_{k n}(\log \vec x)\MINOR{\vec p\\\vec i_n},
\end{align*}
and the result follows by recognizing that as $k$ is increased to $k+1$ in the limit, the sum will contain all the previous terms (corresponding to $\vec p\in P_{kn}$) with the addition of new terms corresponding to $\vec p\in Q_{(k+1)n}$.
\end{proof}

\begin{proof}[Proof of Theorem \ref{thm:generaldetfractionlimit}]
When the summation in Lemma \ref{lem:generaldetlimit} is applied to $g_n(\vec x,\vec at)$the resulting expression will contain factors
\[
V_{k n}(\vec at)^T\MINOR{\vec i_n\\\vec q}= t^{E(\vec q)} V_{k n}(\vec a)^T\MINOR{\vec i_n\\\vec q}.
\]
where
\[
E(\vec q)=\sum_{j=1}^{n}(q_j-1).
\]
The lowest exponent for $t$ will occur exactly once, for $k=n$, when $\vec q=\vec i_n$, and it is
\[
M=E(\vec i_n)=\sum_{j=1}^{n}(j-1)=\frac{n(n-1)}{2},
\]
and by splitting the sum we get
\begin{align*}
 g_n(\vec x,\vec at) &= t^{M} V_n(\vec a)^T \MINOR{\vec i_n\\\vec i_n}
		\cdot D_n \MINOR{\vec i_n\\\vec i_n}
		\cdot V_n(\log \vec x) \MINOR{\vec i_n\\\vec i_n}
		+\mathcal{O}(t^{M+1})
\\& =t^{M} v_n(\vec a) \PAR{\prod_{k=1}^{n}\frac{1}{(k-1)!}}v_n(\log \vec x) + \mathcal{O}(t^{M+1}).
\end{align*}
The final result can now be proven by rewriting the denominator in the theorem as $v_n(\vec at)=t^{M} v_n(\vec a)$, taking the limit, and finally expanding $v_n(\log \vec x)$ by Proposition \ref{prop:vandermondeDeterminant}.
\end{proof}

% morfar
% We express the generalized Vandermonde matrix as a limit involving three growing factors.
% We state the matrix form of Taylor expansion for entries of matrices on the form [f(xi,aj)].
% We present a limit for computing the generalized Vandermonde determinant.
% We present results on the limit of the quotient of the generalized Vandermonde determinant at (x,at) and vn(at)
% We present visualizations supporting this fact.

\section{Conclusions}

From the visualizations in Section \ref{sec:visual3d} it was concluded that the extreme points for the ordinary Vandermonde determinant on the unit sphere in $\mathbb{R}^3$ seems to have some interesting symmetry properties and in Section \ref{sec:analytical} it was proven that extreme points could only appear given certain symmetry conditions, see Remark \ref{rem:r_symmetric} and Theorem \ref{thm:coefficients}. This also allowed visualization of the extreme points of the ordinary Vandermonde determinant on the unit sphere in some higher dimensional spaces, $\mathbb{R}^n$, more specifically for $n = 4,5,6,7$.

The exact location of the extreme points for any finite dimension could also be determined as roots of the polynomial given in Theorem \ref{thm:extreme_points_analytical}. Exact solutions for these roots were also given for the dimensions that were visualized, see Section \ref{sec:analytical_v3} and Section \ref{sec:visualnd}.

Some visual investigation of the generalized Vandermonde matrix was also done in Section \ref{sec:visual3d} but no clear way to find or illustrate where the extreme points was given. The authors intend to explore this problem further.

In Section \ref{sec:power_limit} some limit theorems that involve factorization of a generalized Vandermonde matrix using an ordinary Vandermonde matrix and the ratio between the determinant of a generalized Vandermonde matrix and the determinant of a related ordinary Vandermonde matrix.
% Here we note that we haven't found any results of this type for the optimization of the generalized Vandermonde determinant.
% We note that optimization may be formed on other convex sets than the unit ball.
% We note the usefulness of our method of numerical analysis and conversion to analytical expressions, and state draw-backs.
% not explicit expressions.

\end{document}